\documentclass[a4paper,11pt]{amsart}
\usepackage{graphicx}

\usepackage{amsmath,amsthm}
\usepackage{enumitem}

\newtheorem{theorem}{Theorem}[section]

\newtheorem{lemma}{Lemma}[section]

\theoremstyle{definition}
\newtheorem{definition}{Definition}[section]
\newtheorem{example}{Example}[section]

\theoremstyle{remark}
\newtheorem{remark}{Remark}[section]

\usepackage[cp1250]{inputenc}
\usepackage[nospace,noadjust]{cite}
\usepackage{hyperref}
\usepackage{float}
\usepackage{bbm}
\usepackage{amssymb}
\newcommand{\R}{\mathbb{R}}
\newcommand{\N}{\mathbb{N}}

\newcommand{\diam}{\operatorname{diam}}

\newcommand{\T}{\mathbb{T}}

\newcommand{\F}{\mathcal{F}}

\begin{document}
\title[On triangle-like characterizations]
{On characterization of functions preserving metric-type conditions via triangular and polygonal structures}

\author[F. Turobo\'s ]
{Filip Turobo\'s}
\address[F. Turobo\'s]{
Institute of Mathematics, \L\'od\'z University of Technology,
W\'olcza\'nska 215,
 93-005 \L\'od\'z, Poland}
\email[F. Turobo\'s]{filip.turobos@gmail.com}

\begin{abstract}
Following the train of thought from our previous paper 
we revisit the theorems of Pongsriiam and Termwuttipong by further developing their characterization of certain property-preserving functions using the so-called triangle triplets. We develop  more general analogues of disjoint sum lemmas for broader classes of metric-type spaces and we apply these to extend results of Bors\`ik an Dobo\v{s} 
as well as those obtained by Khemaratchatakumthorn and Pongsriiam. 
 As a byproduct we obtain methods of generating non-trivial and infinite strong $b$-metric spaces which are not metric.
\end{abstract}

\maketitle

\section{Introduction}

The first appearances of metric transforms in the literature date back to 1935, where Wilson \cite{Wilson} has investigated functions preserving the triangle inequality and some other properties. Most researchers in this field agree that the second important paper in this field was due to Sreenivasan, twelve years later \cite{srenivasan}. Additionally, the well-known monograph of Kelley \cite{Kelley} presents some results from this theory as exercises.

The theory of metric preserving functions was developed by Bors\'{i}k, Dobo\v{s} and Piotrowski \cite{BorsikDobos,Doborsik,DobosPiotrowskiRemarks,DoboPiotro,Dobos_Edu,DobosEuclidean}.
See also lectures on this theory by Dobo\v{s} \cite{DobosLectures} and an introductory article by Corazza \cite{Corazza1999}. Of course this list is by no means exhaustive, as many other mathematicians, including Pokorn\'{y} \cite{Pokorny} and Vallin \cite{VallinSubset,VallinDeriv,VallinCont}, had contributed to this topic. 

In 2013, Petru\c{s}el et al. \cite{Petrusel2013} have shown applications of metric transforms in metric fixed point theory.
Another research on this topic which emphasizes the fixed point theory, was simultaneously conducted by Kirk and Shahzad \cite{Kirk2013}, well known for their broad contributions to the field of metric fixed point theory. In their monograph \cite{KirkShahzad}, they highlighted the problem of very few non-trivial, natural examples of strong $b$-metric spaces in the literature. We would like to tackle this issue to some extent in this paper. Also, another follow-up in this topic was made by Pongsriiam and Termwuttipong \cite{Pongsriiam2014}, who generalized the results of Kirk and Shahzad from their paper \cite{Kirk2013}.

This group of Thai mathematicians, namely, Khemaratchatakumthorn, Pongsriiam, Samphavat and Termwuttipong \cite{Tammatada2018,Pongsriiam,wdistances,Pongsriiam2014,Termutti,Samphavat} have made a huge contribution to extension of theory of metric preserving functions. The functions they have investigated were preserving some other (not necessarily metric) axioms of the distance-type functions. 
This general name refers to all semimetrics satisfying a certain condition resembling the triangle inequality.  
In particular, this group of mathematicians  focused mainly on $b$-metric spaces and ultrametric spaces. It is worth noticing that the latter class of spaces was also investigated by other researchers, see \cite{DovgosheyCombinatorial,Dovgoshey}.

Although the concept of a $b$-metric space (also called the quasimetric space) was introduced as early as 1937 by Frink \cite{Frink}, it seemed to be out of favour for quite a long time.  
Recently, $b$-metrics and many other generalizations of the concept of metric space attract many scientists and appear in an increasing number of fields. For a survey on these generalizations and their applications, refer to a paper of Van An et al. \cite{VanAn}. For review of recent progress in the fixed point theory in $b$-metric spaces, see the survey article of Karap\i{}nar \cite{Karapinar}. We also recommend a paper of Cobza\c{s} \cite{Cobzas} and a broad list of references therein.

The purpose of this paper is to extend the already known  characterizations (based on triangle-like structures) of some classes of functions which preserve certain properties of semimetric spaces. 
These extensions will include both strong $b$-metric spaces and the ones satisfying condition known as relaxed polygonal inequality, introduced by Fagin et al. \cite{Fagin}. 
As we approach this issue, we also obtain as a byproduct some equivalents of wide-known results on combining multiple metric spaces into a single one. These supporting lemmas also partially answer the aforementioned issue raised by Kirk and Shahzad.

We then follow with a corresponding result for functions connected with semimetric spaces satisfying the relaxed polygonal inequality. We discuss some of the obtained answers and provide some examples to illustrate both the notions used and theorems obtained.

\section{Preliminaries}
We begin with the definition of a semimetric space, as this is one of the primary concepts we will need to use later on.

\begin{definition}
A \textit{semimetric space} is a non-empty set $X$ equipped with a function $d:X\times X\to [0,+\infty)$, satisfying the following conditions:
\begin{itemize}
    \item[(S1)] For all $x,y\in X$, $d(x,y)=0$ if and only if $x=y$;
    \item[(S2)] For all $x,y\in X$, $d(x,y)=d(y,x)$.
\end{itemize}
Then, the function $d$ is called a \textit{semimetric}.
\end{definition}

The following definition introduces various properties which a semimetric space may possess.

\begin{definition}
Let $(X,d)$ be a semimetric space. If a function $d$ satisfies
\begin{itemize}
\item[(U)] $d(x,z)\leqslant \max\{d(x,y),d(y,z) \}$ for any $x,y,z\in X$,\\ 
		then we call $(X,d)$ an \textit{ultrametric space};
\item[(M)] $ d(x,z)\leqslant d(x,y)+d(y,z)$ for all $x,y,z\in X$,\\ 
		then $(X,d)$ is a \textit{metric space};
\item[(S)] $d(x,z)\leqslant Kd(x,y) + d(y,z)$ for any $x,y,z\in X$,\\ where $K\geqslant 1$ is fixed, then $(X,d)$ is called a \textit{strong} $b$\textit{-metric space};
\item[(P)] $d(x_0, x_n) \leqslant K \cdot \sum_{i=1}^n d(x_{i-1},x_i)$ for any $n\in\mathbb{N}$ and $x_0,x_1,\dots, x_n\in X$ \\ where $K\geqslant 1$ is fixed, then we say that the space $(X,d)$ satisfies the $K$\textit{-relaxed polygonal inequality} ($K$-rpi for short);
\item[(B)] $d(x,z)\leqslant K\left(d(x,y) + d(y,z) \right)$ for any $x,y,z\in X$,\\ where $K\geqslant 1$ is fixed, then $(X,d)$ is called a $b$-\textit{metric space}.
	\end{itemize}
	The coefficients $K$, which appear in definitions (S), (P) and (B) will be called the \textit{relaxation constants}.
\end{definition}

Throughout the article, we will mainly focus on functions which, in a sense, \textit{preserve} or \textit{transform} the aforementioned inequalities (although sometimes at the cost of the topological structure). Therefore, let us introduce an appropriate definition:

\begin{definition}
    Let ($A_1$), ($A_2$) be two properties of a semimetric space. We say that $f:[0,+\infty) \to [0,+\infty)$ is \textit{($A_1$)-($A_2$)-preserving}, if for any semimetric space $(X,d)$ satisfying the property ($A_1$), $f\circ d$ is a semimetric and the space $(X,f\circ d)$ satisfies the condition ($A_2$).
\end{definition}

	For more convenient presentation of inclusions between given classes of functions, we denote by $P_{A_1, A_2}$ the class of all ($A_1$)-($A_2$)-preserving functions. Following this train of thought, by $P_{A_1}$ we will denote the class of all ($A_1$)-($A_1$)-preserving functions. 
	
	Now, let us introduce some definitions connected with \textit{behaviour} of property-preserving functions.
	
	\begin{definition}
	A function $f:[0,+\infty)\to [0,+\infty)$ is said to be 
	\begin{itemize}
	\item \textit{amenable} if $f^{-1}[\{0\}]=\{0\}$;
	\item \textit{subadditive} if $f(a+b)\leqslant f(a)+f(b)$ for all $a,b\in [0,+\infty)$;
	\item \textit{quasi-subadditive} if $f(a+b)\leqslant K\cdot (f(a)+f(b))$ for some $K\geqslant 1$ and all $a,b\in [0,+\infty)$;
	\item \textit{concave} if $f(ta+(1-t)b)\leqslant tf(a)+(1-t)f(b)$ for every $t\in [0,1]$ and $a,b\in [0,+\infty)$; 
	\item \textit{tightly bounded} if there exists $v>0$ such that $f(a)\in[v,2v]$ whenever $a>0$.
	\end{itemize}
	\end{definition}
	
	Below we quote a few results from papers of originators of this topic \cite{BorsikDobos,Corazza1999,DobosLectures} composed into two lemmas, which will be used in the latter part of this paper.
	
\begin{lemma}[\textbf{Sufficient conditions for metric preservation}]\label{huek}
	Let $f:[0,+\infty)\to[0,+\infty)$ be an amenable function. If any of the conditions below is satisfied:
\begin{itemize}
	\item[(i)] $f$ is concave;
	\item[(ii)] $f$ is increasing and subadditive;
	\item[(iii)] $f$ is amenable and tightly bounded;
\end{itemize}
then $f\in P_M$, i.e., $f$ is metric preserving.
\end{lemma}

\begin{lemma}[\textbf{Necessary conditions for metric preservation}]\label{imahusk}
	Let $f\in P_M$. Then $f$ is both amenable and subadditive.
\end{lemma}

For a broader survey of results in the topic of metric preserving functions, see the aforementioned lectures of Dobo\v{s} \cite{DobosLectures} as well as a nice paper summarizing some recent results by Samphavat et al. \cite{Samphavat}.
	
\section{Construction tools for $b$-metric and strong $b$-metric spaces}

It is well known that a disjoint union of two metric spaces, say $(X_1,d_1)$ and $(X_2,d_2)$, is metrizable. In case when we deal with spaces of finite diameter the case is easy. Otherwise, we may replace both metrics by bounded metrics equivalent to them (in particular, this mapping $t\mapsto \frac{t}{t+1}$ is an example of (M)-preserving function).
\[
d_i^\prime (x,y):= \frac{d_i(x,y)}{1+d_i(x,y)} \text{ for } x,y\in X_i, i=1,2.
\]
The resulting spaces obviously are not isometric to the original ones in the general case. Nevertheless, they retain the same topology. In the case of more general distance-type functions, tampering with their values may alter the relaxation constants from their definitions, as the following example depicts. 

\begin{example}
Let $X:=\{1,2,3\}$ and $d$ be a semimetric satisfying:
\[
d(1,2)=d(2,3) = 1 \qquad d(1,3)=4.
\]
It can be easily seen that remaining values of $d$ stem from the axioms (S1) and (S2). A simple calculation also shows that $d$ satisfies condition (B) for $K=2$. 

Now, as previously, let $d^\prime(x,y):= \frac{d(x,y)}{1+d(x,y)}$ for all $x,y\in X$. Thus,
\[
d^\prime(1,2)=d^\prime(2,3) = \frac{1}{2} \qquad d^\prime(1,3)=\frac{4}{5}.
\]
One can immediately notice that such modification of the distances of this simple $b$-metric space resulted in transforming it into a metric space.
\end{example}

\begin{remark}
By changing the original values of $d(1,2)$ and $d(2,3)$ from $1$ to $\frac{1}{4}$, $(X,d)$ would be a $b$-metric space with relaxation constant $K=8$. In such case $d^\prime(1,2)=d^\prime(2,3) = \frac{1}{5}$. 

This means that $(X,d^\prime)$ would no longer be a metric space, but rather a $b$-metric one with relaxation constant $K=2$.
\end{remark}

If we restrict ourselves to the case of finite diameter, then we are able to perform the operation of joining two spaces in such way, that the result satisfies the same type of inequality with the relaxation constant equal to the greater one from the initial spaces.

Let us now put this claim in the formal setting. For a semimetric space $(X,d)$ and its subset $A\subset X$, by $\diam_{d}(A)$ we will denote its diameter with respect to the semimetric $d$, i.e.
\[
\diam_{d}(A) := \sup \{d(x,y)\ : \ x,y\in A \}.
\]
We will skip the subscript in situations where it does not lead to any ambiguities.

\begin{lemma}[\textbf{Concatenation lemma for $b$-metric and strong $b$-metric spaces}]\label{concatenlemma}

	Let $(X_1,d_1)$, $(X_2,d_2)$ be a pair of disjoint, strong $b$-metric spaces ($b$-metric spaces respectively) with relaxation constants $K_1$, $K_2$.
Additionally, assume that both spaces have finite diameter, i.e., $r_1:=\diam_{d_1}(X_1)<\infty$ and $r_2:=\diam_{d_2}(X_2)<\infty$ and $X_1\cup X_2$ have at least three elements. 	
	Let $X:=X_1\cup X_2$. Then, there exists an extension of $d_1,\ d_2$, namely $d:X\times X \to [0,+\infty)$ which is a strong $b$-metric space (respectively $b$-metric space) with relaxation constant $K:=\max \{K_1,K_2\}$ and $\diam_d(X)=\max\{r_1, r_2\}$.
\end{lemma}
\begin{proof}
	We begin by defining $d$. Put
	\begin{equation}\label{definitionConcatenation}
	d(x,y):=\begin{cases}
	d_1(x,y), & x,y\in X_1\\
	d_2(x,y), & x,y\in X_2\\
	\frac{\max\{r_1,r_2\}}{1+K}, & \text{everywhere else.}
	\end{cases}
	\end{equation}
	From \eqref{definitionConcatenation} we may infer that $d_{|X_i\times X_i} = d_i$ for $i=1,2$, thus it is an extension of both $d_1$ and $d_2$. The definition is correct thanks to $X_1$ and $X_2$ being disjoint.
	
	It is easily seen that the pair $(X,d)$ satisfies the conditions of a semimetric space, as having at least three elements guarantees that $\max\{r_1,r_2\}>0$. Moreover, the fact, that the resulting space diameter equals the greater one of the diameters of composing spaces is a direct consequence of the definition of $d$.  
	What is left to check is whether the strong $b$-metric (or just $b$-metric) condition holds. Consider any three distinct points $x,y,z\in X$. We have the following cases:
	\begin{enumerate}
		\item $x,y,z\in X_1$ or $x,y,z\in X_2$. In such cases $d$ coincides with either $d_1$ or $d_2$, so the desired inequality holds.
		\item $x,y\in X_1$ and $z \in X_2$. We need to check three variants of a strong $b$-metric inequality.
		\begin{equation*}
		\begin{aligned}
			d(x,y) &=& d_1(x,y) \leqslant r_1 &\leqslant& \frac{\max\{r_1,r_2\}}{1+K} \cdot (1+K)  = Kd(x,z)+d(z,y),\\
			d(x,z) &=& \frac{\max\{r_1,r_2\}}{1+K} &\leqslant& K\cdot 0 + \frac{\max\{r_1,r_2\}}{1+K} \leqslant K d_1(x,y)+d(y,z),\\
			d(y,z) &=& \frac{\max\{r_1,r_2\}}{1+K} &\leqslant& K\cdot 0 + \frac{\max\{r_1,r_2\}}{1+K} \leqslant K d_1(y,x) + d(x,z).\\
		\end{aligned}
		\end{equation*}
		One can easily notice, that the $b$-metric inequality follows from this reasoning as well.
		\item $x\in X_1$ and $y,z\in X_2$. The desired inequalities in this case are obtained analogously as above. 
	\end{enumerate}
	Therefore, $(X,d)$ is a strong $b$-metric space (or $b$-metric space).
\end{proof}

\begin{remark}
	In the case where we are interested in concatenating $b$-metric spaces, we may replace $1+K$ with $2K$, which yields a shorter distance between these two parts of space.
\end{remark}

Now we know that in a case when we have a finite family of semimetric spaces satisfying a given property, their disjoint sum also does satisfy such condition. The problem appears when we want to concatenate an infinite family of metric spaces. As we will see in the following results, sometimes it is desired to construct an unbounded space, but the lemma above does not offer a help in such case either. Therefore, we need to introduce the following

\begin{lemma}[\textbf{On the union of an increasing sequence of semimetric spaces}]\label{summationLemma}

 	Let $\T$ be a linearly ordered set and $\F := \left\{(X_t,d_t) \; : \; t\in\T\right\}$ be an increasing family of strong $b$-metric ($b$-metric respectively) with a fixed constant $K\geqslant 1$ spaces i.e. for every pair of indices $s\leqslant t$ we have $X_s \subseteq X_t$ and $d_s \subseteq d_t$ (which means $d_s(x,y) = d_t(x,y)$ for $x,y\in X_s$). Then a pair $(X,d)$ defined as 
\[
X:=\bigcup_{t\in \T} X_t \qquad \qquad d(x,y):=d_p(x,y), \text{ where } p \text{ is any index such that } x,y\in X_p
\] 
is strong $b$-metric ($b$-metric respectively) with the relaxation constant $K\geqslant 1$.
\end{lemma}
\begin{proof}
Since $\F$ is increasing then $d(x,y)$ is well defined for all $x,y\in X$.
We can then proceed to checking axioms of semimetric
\begin{itemize}
\item[(S1)] $d(x,y)=0$ implies that $d_{t_0} (x,y)=0$ for some $t_0\in \T$. Therefore $x=y$. The reverse implication is obvious;
\item[(S2)] $d(x,y)=d_t(x,y)=d_t(y,x)=d(y,x)$ for some $t\in \T$.
\end{itemize} 
Lastly, we need to prove the (S) condition (or (B) condition) for $(X,d)$. Take any three distinct points $x,y,z\in X$. There exists $t_0\in \T$ such that $x,y,z\in X_{t_0}$. Since $(X_{t_0},d_{t_0})$ satisfies (S) (or (B), respectively), we have
\[
d(x,z) = d_{t_0} (x,z)\leqslant K d_{t_0} (x,y) + d_{t_0}(y,z) = Kd(x,y)+d(y,z).
\]
Consequently, $(X,d)$ satisfies (S) condition. The proof of $b$-metric case goes analogously.
\end{proof}

\begin{remark}
It is worth noting that we do not employ this lemma in its full extent, as in the subsequent reasoning we will restrict to the cases where $\T=\N$. 
\end{remark}

Notice that these two tools allow us to construct an increasing sequence of strong $b$-metric spaces having some desired property and then gluing them into a single, larger space -- perhaps of infinite diameter.

Although strong $b$-metric spaces constructed in such a way can hardly be seen as natural, this approach partially refers to \cite[Remark 12.4]{KirkShahzad} of Kirk and Shahzad on providing more examples of strong $b$-metric spaces.

\section{Characterizing certain property-preserving functions}

In this section we introduce the main result of this paper, foreshadowed in the abstract.
However, to proceed with introducing the new result, we need to extend \cite[Definition 3]{Tammatada2018}.
	
\begin{definition}
	Let $K\geqslant 1$ and $a,b,c \in [0,+\infty)$ be such that $a\geqslant b \geqslant c$. We will say that the $(a,b,c)$ forms a \textit{triangle triplet} (respectively $K$-triangle triplet or strong $K$-triangle triplet) if the respective set of inequalities: (TT), ($K$-TT) or (S$K$-TT), holds:
	    \begin{itemize}
	    \item[(TT)] $a\leqslant b+c$;
	    \item[($K$-TT)] $a\leqslant K\left( b+c\right)$;
	    \item[(S$K$-TT)] $a\leqslant K c+b$. 
	    \end{itemize}
\end{definition}

\begin{remark}\label{remaaaark}
	Usually, no assumptions on relations between $a,b$ and $c$ are made. Adding these relations enables us to shorten the latter part of the definitions as well as shorten some parts of the proofs. While trying not to confuse the reader, we will call a triplet $(a,b,c)$ a triangle triplet etc. if rearrangement of those three numbers form the respective variant of the triangle triplet. 
	
	For example we will refer to $(2,4,5)$ as the triangle triplet due to the fact, that $(5,4,2)$ satisfies the discussed definition.
\end{remark}

We shall now quote a well-known result from the theory of metric preserving functions, which we will try to replicate \cite{BorsikDobos,Corazza1999,DobosLectures}.

\begin{lemma}[\textbf{Triangle triplet characterization of $P_M$}]

Let $f:[0,+\infty)\to[0,+\infty)$ be an amenable function. Then $f$ is metric preserving iff for any triangle triplet $(a,b,c)$, values $\{f(a),f(b),f(c)\}$ can be arranged into a triangle triplet. 
\end{lemma}

Let us focus for a moment on the particular subclass of functions preserving properties (B) and (S), where the relaxation constant on the original space is fixed. 

\begin{lemma}[\textbf{Triangle-like triplet characterization of property-preserving functions with fixed relaxation constant}]\label{mainlemma}

Let $f:[0,+\infty)\longrightarrow [0,+\infty)$ be a function such that for every (strong) $b$-metric space $(X,d)$ with fixed relaxation constant $K\geqslant 1$, the space $(X,f\circ d)$ is a (strong) $b$-metric space. Then, there exists $K^\prime \geqslant 1$ such that:
\begin{itemize}
\item[a)] for every (strong) $K$-triangle triplet $(a,b,c)$, the values $f(a)$, $f(b)$ and $f(c)$ can be arranged into a (strong) $K^\prime$-triangle triplet.
\item[b)] the relaxation constant of the resulting space $(X,f\circ d)$ is bounded by $K^\prime$.
\end{itemize}
\end{lemma}

\begin{remark}
This lemma actually gives us insight into six families of property-preserving functions, namely $P_{B}$, $P_{B,S}$, $P_{S,B}$, $P_{S}$, $P_{M,B}$ and $P_{M,S}$ (the last two obtained by fixing $K=1$). Also, one can also immediately deduce b) from a).
\end{remark}

\begin{proof}
Due to the fact, that each version of this lemma is proved by the same method, we will simply prove the version for particular subclass of $P_{B,S}$.

Let $f$ be a function satisfying the assumptions of our lemma, i.e. for every $b$-metric space $(X,d)$ with relaxation constant $K\geqslant 1$ we have that $(X,f\circ d)$ is a strong $b$-metric space.

Suppose the contrary, i.e., for every $K^\prime \geqslant 1$ (in particular for every $n\in\mathbb{N}$) there exists a $K$-triangle triplet $\{a_n,b_n,c_n\}$ which is mapped to $\left\{f(a_n),f(b_n),f(c_n)\right\}$, the latter one not forming a strong $n$-triangle triplet. Without the loss of generality, we may assume that $f(a_n)>n\cdot f(b_n)+f(c_n)$ (this fact becomes clear in the further part of the proof).
	
	Let $\hat{X}_1 = X_1:=\{(1,1),(1,2),(1,3)\}$ and \[
	d_1((1,1),(1,2))=a_1 \qquad d_1((1,2),(1,3))=b_1,  \qquad d_1((1,1),(1,3))=c_1. \] The rest of the values, again, stems from extending $d_1$ to the whole $X_1\times X_1$ as a semimetric (i.e., (S1) and (S2) conditions imply the remaining function values), yielding $\hat{d}_1$. Thus $(\hat{X}_1,\hat{d}_1)$ is not a strong $b$-metric space with relaxation constant $1$. Nonetheless, it is still a strong $b$-metric space if we replace $1$ by some larger constant.
	
    Now, we will proceed inductively. Having defined $\left(\hat{X}_{n-1},\hat{d}_{n-1}\right)$, which is a finite (thus bounded) $b$-metric space with relaxation constant $K$, we try to define $\hat{X}_{n}$. Put
	\(
    X_n:=\{(n,1),(n,2),(n,3)\} \) and  \(
    d_n:X_n\times X_n \to [0,+\infty),
    \)
    as a unique semimetric satisfying 
    \[
    \hat{d_n}((n,1),(n,2)) = a_n, \qquad
    \hat{d_n}((n,2),(n,3)) = b_n, \qquad
    \hat{d_n}((n,1),(n,3)) = c_n.
    \]
    Again, $(X_n,d_n)$ is a bounded $b$-metric space with relaxation constant $K$. We can now use Lemma \ref{concatenlemma} to obtain larger $b$-metric space $\hat{X}_n:=X_{n}\cup \hat{X}_{n-1}$, equipped with the extended $b$-metric $\hat{d}_{n}$, which equals $d_n$ on $X_n\times X_n$ and $\hat{d}_{n-1}$ on $\hat{X}_{n-1}\times\hat{X}_{n-1}$. 
    Successively, we obtain the increasing family of $b$-metric spaces $\left\{\left(\hat{X}_n,\hat{d}_n\right) \ : \ n\in\mathbb{N} \right\}$. Due to Lemma \ref{summationLemma} we obtain a $b$-metric space $(X,D)$ defined as
    \[
    X:=\bigcup_{n\in\mathbb{N}} X_n, \qquad D(x,y):=\hat{d}_{\min\{p \; : \; x,y\in\hat{X}_p\}}(x,y).
    \]
    
    What is left to prove is that $(X,f\circ D)$ fails to be a strong $b$-metric space.
    
    Suppose that (S) inequality holds in $(X,D)$ for some $K^\prime \geqslant 1$ and let $n_0\geqslant K^\prime$ be a fixed natural number. This way, the fact that $(X,D)$ satisfies the strong $b$-metric inequality with constant $K^\prime$ implies satisfying it for $n_0$. Considering $(n_0,1),(n_0,2),(n_0,3)\in X$ yields
    \begin{eqnarray*}
    D\left((n_0,1),(n_0,2)\right) &=& d_{n_0}\left( (n_0,1),(n_0,2)\right) = a_{n_0}\\ &>& n_0 \cdot b_{n_0}+c_{n_0} = n_0 \cdot D\left((n_0,2),(n_0,3)\right) +  D\left((n_0,3),(n_0,1)\right)\\
    &=& n_0 \cdot d_{n_0}\left( (n_0,2),(n_0,3)\right) + d_{n_0}\left( (n_0,3),(n_0,1)\right).    
    \end{eqnarray*}
	Of course $f(a_n)$, $f(b_n)$ and $f(c_n)$ do not have to be aligned in the same order as $a_n$, $b_n$ and $c_n$, but the result stays the same as we can interchange the appropriate pairs of points in the reasoning above to obtain the proper counterexample.

    This shows that $(X,f\circ D)$ fails to satisfy strong $b$-metric condition for $n_0$, thus it fails for $K^\prime$ as well. Since $K^\prime$ was arbitrary, $(X,f\circ D)$ is not a strong $b$-metric space despite $(X,D)$ being a $b$-metric one. This concludes the proof of both part a) and b), since it yields a contradiction with the assumption that $(X,f\circ d)$ is a strong $b$-metric space whenever $(X,d)$ is a $b$-metric space with relaxation constant $K$. 
    
    The reasoning for the remaining variants of this lemma is exactly the same -- all that needs to be changed is the type of triangle triplets used in constructing the proper counterexample.
\end{proof}

The said lemma allows us to draw somewhat surprising conclusion, that a single property-preserving mapping $f$ does not allow us to obtain arbitrarily large values of relaxation constant on resulting space whenever the relaxation constant of the initial space is 
bounded. This seems particularly interesting for classes $P_{MS}$ and $P_{MB}$.

We shall now proceed with first of the main results of this paper, which is an extension of analogous characterization already proven by other authors (see \cite{BorsikDobos,Corazza1999,Tammatada2018}). Their results have been incorporated in this theorem in points (i) and (ii) -- in particular, we will not provide the proofs for those, as they were well-described in the respective papers they were taken from.

\newpage

\begin{theorem}[\textbf{Characterization of triangle-type inequality preserving functions via triangle-like triplets:}] \label{extendedViaTriangles}
	Let $f:[0,+\infty)\to [0,+\infty)$ be an amenable function. Then:
	\begin{enumerate}
		\item $f\in P_{MB}=P_{B}$ $\iff$ there exists $K\geqslant 1$ for every triangle triplet $(a,b,c)$, the values $\left(f(a),f(b),f(c)\right)$ form a $K$-triangle triplet.
		\item $f\in P_{M}$ $\iff$ for any triangle triplet $(a,b,c)$, the resulting values $\left(f(a),f(b), f(c)\right)$ form a triangle triplet as well.
		\item $f\in P_{BS}$ $\iff$ for every $K_1\geqslant 1$ there exists $K_2\geqslant 1$, such that for every $K_1$-triangle triplet $(a,b,c)$, the values $\left(f(a),f(b), f(c)\right)$ form a $K_2$-strong triangle triplet.
		\item $f\in P_{S}$ $\iff$ for every $K_1\geqslant 1$ there exists $K_2\geqslant 1$, such that for every $K_1$-strong triangle triplet $(a,b,c)$, the values $\left(f(a),f(b), f(c)\right)$ form a $K_2$-strong triangle triplet.
		\item $f\in P_{SB}$ $\iff$ for every $K_1\geqslant 1$ there exists $K_2\geqslant 1$, such that for every $K_1$-strong triangle triplet $(a,b,c)$, the values $\left(f(a),f(b), f(c)\right)$ form a $K_2$-triangle triplet.
		\item $f\in P_{MS}$ $\iff$ there exists $K\geqslant 1$, such that for every triangle triplet $(a,b,c)$, the values $\left(f(a),f(b), f(c)\right)$ form a $K$-strong triangle triplet.
		\item $f\in P_{SM}$ $\iff$ for every $K$-strong triangle triplet $(a,b,c)$, the values $\left(f(a),f(b), f(c)\right)$ form a triangle triplet.
	\end{enumerate}
\end{theorem}
\begin{proof}
	Once again, we would like to remind that $f$ does not have to be monotone, thus the relation $f(a)\geqslant f(b) \geqslant f(c)$ does not have to hold (see Remark \ref{remaaaark}).

	The necessity in each case follows as a simple conclusion from Lemma \ref{mainlemma}. In the last, seventh equivalence another short proof of necessity can be given, as considering a three-point space $X:=\{1,2,3\}$ is enough. 
	
	Supposing that for some $K$ a strong $K$-triangle triplet $\{a,b,c\}$ which is not mapped to a triangle one exists yields an instant contradiction if we define
  \[
  d(1,2):=a, \ d(2,3):=b, \ d(3,1):=c
  \]
  and allow the assumptions (S1), (S2) to fill in the rest of values. Applying $f$ to such  defined strong $b$-metric space maps it into a semimetric structure which is not a metric space. The sufficient part of this last implication follows from the same method as described in point f. above. 
    
 	A sufficiency part can be shown as follows.
 	
 	Assume that $f$ satisfies the right-hand side of equivalence (3) and let $(X,d)$ be an  arbitrary $b$-metric space. Let us denote the relaxation constant of this space by $K$. Consider three distinct points $x_1,x_2,x_3\in X$. WLOG we can assume that
 	\[
 	f\circ d \left( x_1,x_3\right) \geqslant f\circ d \left( x_1,x_2\right) \geqslant f\circ d \left( x_2,x_3\right).
 	\]
 	Since $(X,d)$ is a $b$-metric space, then values $d \left( x_1,x_3\right), d \left( x_1,x_2\right), d \left( x_2,x_3\right)$ can be arranged into a $K$-triangle triplet. From our assumption there exists $K^\prime$ such that $\{f\circ d \left( x_1,x_3\right),f\circ d \left( x_1,x_2\right),f\circ d \left( x_2,x_3\right) \}$ is a strong $K$-triangle triplet. From the definition of such triplet we obtain
 	\[
 	f\circ d \left( x_1,x_3\right) \leqslant K^\prime \cdot f\circ d \left( x_2,x_3\right)+ f\circ d \left( x_1,x_2\right).
 	\]
   Due to $K^\prime$ being independent of the choice of $x_1,x_2,x_3$ (since it depends solely on $K$) the proof of equivalence (3) is finished.
   
The sufficiency proofs for points (4-7) are almost identical.   
\end{proof}


This characterization might allow us to construct other results describing the properties of metric-type property preserving functions. Before moving to the next part of this paper, we point out two issues we are concerned about.

First one of them is the observation that for a fixed function $f$, the right-hand side of each of the implications (3), (4) and (5) in fact define a mapping $g:[1,+\infty)\to [1,+\infty)$, which for every $K_1\geqslant 1$ assigns the smallest possible $K_2\geqslant 1$ such that $(X,f\circ d)$ is $K_2$-parametrized space, whenever $(X,d)$ is the space of appropriate type with relaxation constant $K_1$. We know that in all three cases $g(1) = 1$ and the function $g$ is obviously nondecreasing. It is reasonable to think, that different functions $f_1$, $f_2$ are described by the same mapping $g$, but an unanswered question remains -- what properties of $f$ can we infer based on $g$ and conversely? Moreover, we would like to know what kinds of functions $g$ can be obtained in this way i.e., what attributes do we require from such functions to be obtainable from the discussed theorem.

Theorem \ref{extendedViaTriangles} also shifts the research on this functions from metric and topological setting to more real-function theoretical field, connected with functional inequalities. We already know several of the inclusions between the listed above $P_{A_1, A_2}$ families. Hopefully, this characterization will allow us to explore this topic further, allowing us to prove equalities or proper inclusions for some families of property-preserving functions. Thus, we pose the following

\textbf{Open problem 1:} Establish the relations between $P_{S,M}$, $P_{S}$, $P_{S,B}$, $P_{B,S}$ and other classes of property preserving functions. In particular, determine which of the inclusions are proper.

\section{Further characterizations for semimetrics satisfying relaxed polygonal inequalities}

It turns out, that the reasoning that was applied to obtain the Theorem \ref{extendedViaTriangles} can be extended on functions satisfying (P) axiom. However, we need to introduce additional definitions.

\begin{definition}\label{defpolygon}
A finite tuple of non-negative real numbers $A:=(a_1,\dots,a_n)$ such that 
\[
a_1 \geqslant a_2 \geqslant \dots \geqslant a_n.
\]
is said to be a $K$\textit{-relaxed polygon} if the following inequality holds:
\[
a_1 \leqslant K \cdot \sum_{i=2}^{n} a_i .
\]
\end{definition}

\begin{remark}\label{remmar}
Analogously to the Remark \ref{remaaaark}, we will refer to a tuple $(a_1,\dots,a_n)$ as $K$\textit{-relaxed polygon} whenever its elements, sorted in nonincreasing order, satisfy Definition \ref{defpolygon}.
\end{remark}

\begin{example}
For example, a tuple $A:=(120, 20, 10, 10, 10, 10)$ is an example of $K$-relaxed polygon with $K=2$. Indeed
\[
120 \leqslant 2\cdot \left( 20 +10 +10 +10 +10\right) = 120.
\]
Notice, that we do not require the largest number to be bounded by any subset sum of the remaining values, i.e. for example we do not require $120$ to be bound from above by $K\cdot (10+10)$.
\end{example}

One can observe, that such a tuple $A$ is a $K$-relaxed polygon if and only if
\[
\left(1+K\right)\cdot a_1 \; \leqslant \; K\cdot \sum_{i= 2}^{n} a_i.
\]

Then, we introduce the following notion of implementation.
\begin{definition}
Consider a $K$-relaxed polygon \[
A:=(a_1,\dots,a_n), \qquad a_1 \geqslant \dots \geqslant a_n
\] and a semimetric space $(X,d)$ satisfying $K$-rpi. If there exists a sequence of $n$-points, $x_1,\dots,x_n\in X$ such that 
\begin{itemize}
\item[(I1)] $d(x_1,x_n):=a_1$;
\item[(I2)] for each $2\leqslant i \leqslant n$, $d(x_{i-1},x_{i}):=a_{i}$;
\end{itemize}
then we say that $(X,d)$ \textit{implements} a $K$-relaxed polygon $A$.
\end{definition}

One may wonder if for any $K$-relaxed polygon there exists a semimetric space $(X,d)$ which implements it and the answer to that question is, luckily, positive.

\begin{lemma}[\textbf{Implementation lemma:}]\label{implement}

Let $A:=(a_1,\dots,a_n)$ be a $K$-relaxed polygon. There exists an $n+1$ element set $X$ with a semimetric $d$ which satisfies $K$-rpi and implements $A$.
\end{lemma}
\begin{proof}
We will construct such space in two steps. Let $X:=\{1,\dots, n\}$ and put 
\begin{itemize}
\item $d(1,n):=a_1$;
\item for each $1\leqslant i \leqslant n$, $d({i-1},{i}):=a_{i}$;
\end{itemize}
Thus, all of the requirements (I1)-(I2) are satisfied. This definition leaves us with $n$-gon with no diagonal distances defined. Consider now $i,j\in X$ such that they are not adjacent i.e. $|i-j|\notin\{1,n\}$. 

For such pair (for the sake of simplicity let us assume $i<j$) one can define the distance as follows
\begin{equation}\label{functiondef}
d(i,j):=\min\left\{ \sum_{k=i}^{j-1} d(k,k+1), \left(\sum_{k=j}^{n-1} d(k,k+1) +  d(1,n) + \sum_{k=1}^{i-1}d(k,k+1) \right) \right\},
\end{equation}
where the sums in the latter part of the minimum might be empty (for example when $i=1$ and $j<n$ or $j=n$ and $i>1$). 

\begin{figure}[H]
\includegraphics[width=\textwidth]{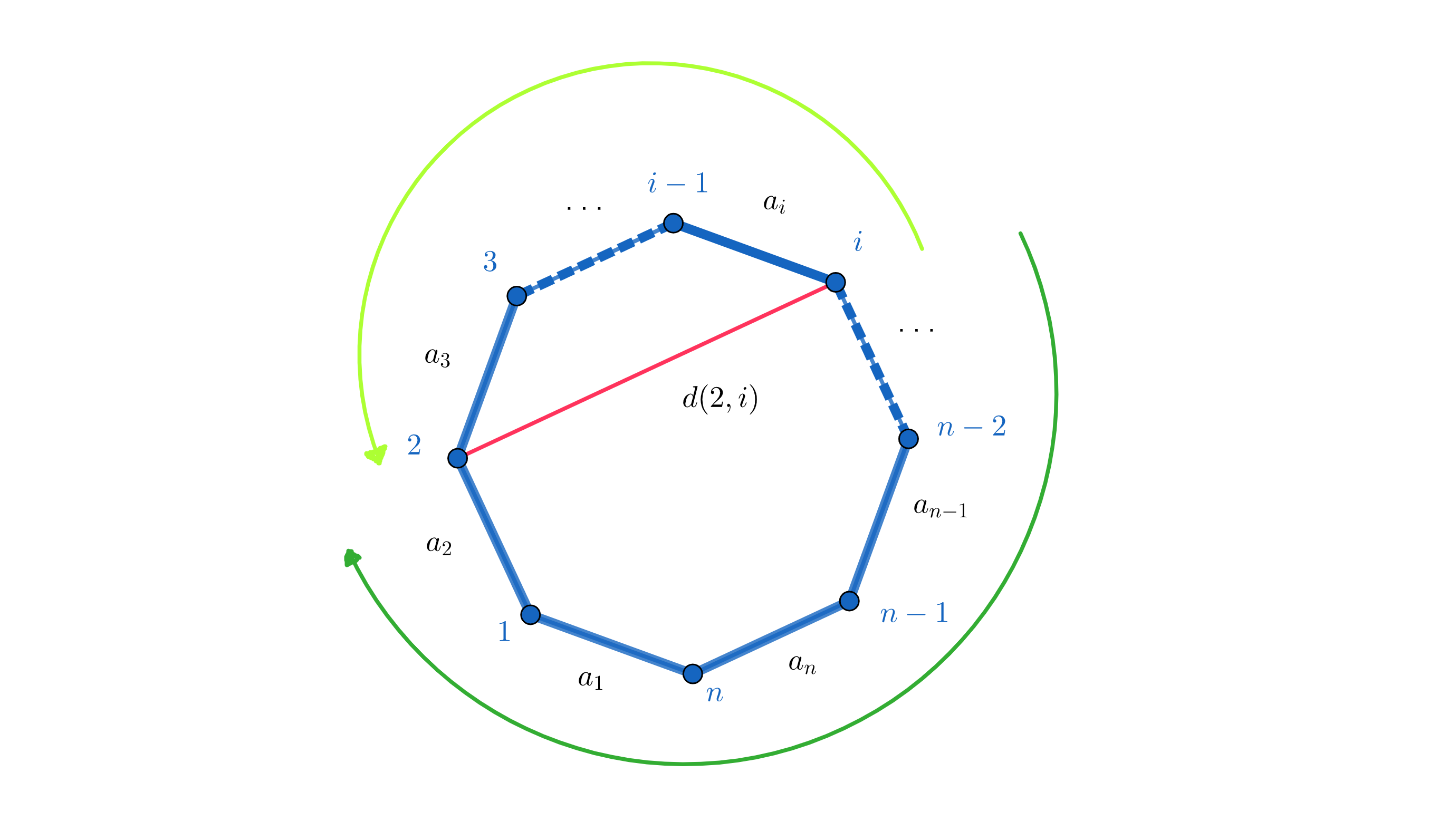}
\caption{Informally, the distance between two points (for example, red line between $2$ and $i$) can be described as the shorter of the paths over the perimeter. Notice, that such distance does not always have to be the left-oriented path (here denoted by lime green).} 
\end{figure}

Notice, that for any two non-adjacent points $i,j$ thanks to the formulation of $d(i,j)$ the distance can be bound from above by the distance over the perimeter of the discussed polygon. Indeed, for any sequence of point $i,x_1,x_2,\dots,x_k,j\in X$.
{\small
\begin{eqnarray*}
K\cdot \left( d(i,x_1) + \dots + d(x_k,j)\right) &\geqslant & K\cdot \min\left\{ \sum_{k=i}^{j-1} d(k,k+1), \left( d(1,n) + \sum_{k=j}^{n-1} d(k,k+1) + \sum_{k=1}^{i-1}d(k,k+1) \right) \right\}\\ &\geqslant &  d(i,j).
\end{eqnarray*}
}

Perhaps this part of the proof can be reworked to be nicer. 
Now, consider any two adjacent points of the space, $i$ and $i+1$ (the case of $n$ and $1$ is analogous). 
Consider any point $x\in X$. WLOG we can assume that $x>i$. 
Then we have three posibilities:
\begin{enumerate}
\item $d(i,x) = \sum_{k=i}^{x-1} d(k,k+1)$, then obviously 
$d(i,x)\geqslant d(i,i+1)$;
\item $d(i,x) = d(1,n) + \sum_{k=x}^{n-1} d(k,k+1) + \sum_{k=1}^{i-1}d(k,k+1)$.\\
If $d(x,i+1) =d(1,n) + \sum_{k=x}^{n-1} d(k,k+1) + \sum_{k=1}^{i}d(k,k+1)$ then $d(x,i+1)\geqslant d(i,i+1)$ and the case is once again trivial.
\item $d(i,x) = d(1,n) + \sum_{k=x}^{n-1} d(k,k+1) + \sum_{k=1}^{i-1}d(k,k+1)$ and $d(x,i+1) = \sum_{k=i+1}^{x-1} d(k,k+1)$.

Then, using the fact that $A$ is a $K$-relaxed polygon
\[
d(i,i+1) \leqslant K\cdot \left( d(1,n) - d(i,i+1) + \sum_{i=1}^n d(k,k+1)\right),
\]
which translates to 
\[
d(i,i+1) \cdot(1+K) \leqslant a_1 \leqslant K\cdot \sum_{k=1}^n d(k,k+1) = K\cdot \sum_{k=1}^n a_k.
\]
\end{enumerate}

When more points are considered, the reasoning remains the same (although it might need some more formal touch) -- we either use the fact that length of one of the edges exceeds the distance $d(i,i+1)$ or the sum of the distances is greater or equal to the remaining part of the perimeter of the initial $K$-relaxed polygon. In any case, the property (P) from Definition 2. holds.
\end{proof}

To obtain analogous result to Theorem \ref{extendedViaTriangles} we need lemmas which are analogous to Lemmas \ref{concatenlemma} and \ref{summationLemma}. We provide only the short sketches of proofs to the lemmas, since the proofs are similar to the respective results from the previous sections.

\begin{lemma}[\textbf{Concatenation lemma for semimetric spaces satisfying $K$-rpi}]\label{ConcatenationForSpaces}
	
	Let $(X_1,d_1)$, $(X_2,d_2)$ be a pair of disjoint, semimetric spaces satisfying property (P) with relaxation constants $K_1$, $K_2$. Assume that $\diam_{d_1}(X_1)=r_1<\infty$ and $\diam_{d_2}(X_2)=r_2<\infty$.	
	Let $X:=X_1\cup X_2$. There exists an extension of $d_1,\ d_2$, namely $d:X\times X \to [0,+\infty)$ which satisfies $K$-rpi, where $K:=\max \{K_1,K_2\}$. Additionally, $\diam_d(X)=\max\{r_1, r_2\}$.
\end{lemma}
\begin{proof}
Let us begin with defining $d$ as 
	\begin{equation}\label{definitionConcatenation2}
	d(x,y):=\begin{cases}
	d_1(x,y), & x,y\in X_1\\
	d_2(x,y), & x,y\in X_2\\
	\frac{\max\{r_1,r_2\}}{2K}, & \text{everywhere else,}
	\end{cases}
	\end{equation}
where $K:=\max\{r_1,r_2\}$.

One can immediately notice, that such space $(X,d)$ will have its diameter at most $r:=\max\{r_1,r_2\}$. What is left to prove is that $d$ fulfills the property $K$-rpi. Let $n\in\N$ and $x_1, \dots, x_n\in X$.

The cases where both $x_1$ and $x_n$ belong to the same $X_i$ are relatively obvious, since in the cases where all remaining $x_k$, $1\leqslant k <n$ belong to the same component, the respective relaxed polygonal property steps in. At the same time, if there is at least one $x_{k_1}$ from the second component, then
\[
d(x_1,x_n)\leqslant r_i \leqslant 2K \cdot \frac{r}{2K} \leqslant K\cdot \sum_{i=1}^{n} d(x_{i-1},x_i).
\]
The last inequality follows from the fact, that the polygon consisting of points $x_1,\dots,x_n$ consists of at least two indices, $1\leqslant i_0,i_1 < n$, for which $x_{i_0}\in X_i$ and $x_{i_0+1}\notin X_i$, as well as $x_{i_1}\notin X_i$ and $x_{i_1+1}\in X_i$.

If $x_1, x_n$ belong to distinct components, then there exists $1\leqslant k< n$ such that $x_k$ belongs to the same component as $x_1$ and $x_{k+1}$ does not. Therefore
\[
\frac{r}{2K} = d(x_1,x_n) \leqslant K\cdot d(x_k,x_{k+1}) \leqslant K\cdot \sum_{i=1}^n d(x_{i-1},x_i).
\]
This finishes the proof of the fact that $(X,d)$ satisfies $K$-rpi.
\end{proof}

\begin{lemma}[\textbf{Summation lemma for semimetric spaces satisfying $K$-rpi}]\label{SummationForPSpaces}

Let $\T$ be a linearly ordered set and $\F := \left\{(X_t,d_t) \; : \; t\in\T\right\}$ be an increasing family of semimetric spaces satisfying $K$-rpi with a fixed constant $K\geqslant 1$. Then a semimetric space $(X,d)$ given by 
\[
X:=\bigcup_{t\in \T} X_t \qquad \qquad d(x,y):=d_p(x,y), \text{ where } p \text{ is any index s.t. } x,y\in X_p
\] 
satisfies $K$-rpi as well.
\end{lemma}
\begin{proof}
Since (P) property considers only finite sequences of points (albeit of arbitrary length), then the proof of this lemma is almost exactly the same as of the Lemma \ref{summationLemma}.
\end{proof}

Using these definitions and lemmas, one can formulate and prove the following theorem, analogous to the Theorem \ref{extendedViaTriangles}.

\begin{theorem}[\textbf{Characterization of (P)-preserving mappings}]\label{SecondMainTheorem}

Consider an amenable function $f:[0,+\infty)\to [0,+\infty)$. Then $f$ is (P)-preserving iff for any $K_1\geqslant 1$ there exists $K_2$ such that $f$ maps any $K_1$-relaxed polygon to $K_2$-relaxed polygon, i.e.
\begin{eqnarray}
\forall_{K_1\geqslant 1} \ \exists_{K_2 \geqslant 1} \ \Bigg( \forall_{n\in\N} \ \forall_{a_1,\dots,a_n \in [0,+\infty)} \
\left( \ \left(1+K_1\right)\cdot \max\{a_i \ : \ i\leqslant n \} \; \leqslant \; K_1\cdot \sum_{1\leqslant i\leqslant n}a_i \ \right)\nonumber  
\\\label{K1K2condition}
\implies \left( \left(1+K_2\right)\cdot \max\{f(a_i) \ : \ 	i\leqslant n \} \; \leqslant \; K_2\cdot \sum_{1\leqslant i\leqslant n}f(a_i) \right) \ \Bigg).
\end{eqnarray}
\end{theorem}
\begin{proof}
We begin by showing the sufficiency of the proposed condition. Assume that \eqref{K1K2condition} holds. Let $(X,d)$ be any semimetric space satisfying $K_1$-relaxed polygonal inequality. Consider a finite sequence of points $x_1,\dots, x_n\in X$. Since $d$ satisfies $K_1$-rpi, then the tuple
\[A:=\left( d\left(x_1,x_n\right),  d\left(x_1,x_2\right),d\left(x_2,x_3\right), \dots, d\left(x_{n-1},x_{n}\right) \right)\]
is (up to an order of elements) a $K_1$-relaxed polygon. Therefore, the values of $f$ at these  respective tuple
\[
A_f:=\left(f\circ  d\left(x_1,x_n\right), f\circ  d\left(x_1,x_2\right),f\circ  d\left(x_2,x_3\right), \dots, f\circ d\left(x_{n-1},x_{n}\right) \right)
\]
can be arranged into a $K_2$-relaxed polygon. This implies that
\begin{eqnarray*}
f\circ d (x_1,x_n) &\leqslant & \max A_f \leqslant K_2 \cdot \left( f\circ d (x_1,x_n) +\sum_{i=2}^n f\circ d\left(x_{i-1},x_{i}\right)\right) - K_2 \max A_f\\
&\leqslant & K_2 \cdot \left( f\circ d (x_1,x_n) +\sum_{i=2}^n f\circ d\left(x_{i-1},x_{i}\right)\right) - K_2 f\circ d (x_1,x_n)\\
&=& K_2 \cdot  \sum_{i=2}^n f\circ d\left(x_{i-1},x_{i}\right),
\end{eqnarray*} 
so once again we are able to observe that the particular order in which the elements appear in $A_f$ does not matter as long as the condition (P) holds.

Since the points $x_1,\dots,x_n$ were chosen arbitrarily, $(X,f\circ d)$ satisfies $K_2$-rpi.

The proof of the other implication goes analogously to the scheme from the proof of Theorem \ref{extendedViaTriangles}. Fix function $f$. If we suppose the contrary, we obtain the existence of such $K_1\geqslant 1$ that for every $K_2\geqslant 1$ there exists some $K_1$-relaxed polygon which is not mapped to a $K_2$-relaxed polygon (albeit again it might be a $K_2'$-relaxed polygon for some greater constant $K_2'$). Consider a sequence of such $K_1$ polygons $A_n$. Due to Lemma \ref{implement}, this implies the existence of spaces $(X_n,d_n)$, all of which satisfy $K_1$-rpi. Therefore, 
using Lemma \ref{ConcatenationForSpaces} we obtain an increasing sequence of spaces with (P)-property with the same constant $K_1$. What remains is to sum them up using Lemma \ref{SummationForPSpaces}, obtaining a large space $(X,d)$ which satisfies $K_1$-rpi as well. The function $f$ fails to preserve property (P) due to the fact that for any natural number $n$ there exists a set of points $\{x_1,\dots,x_k\}$ (coming from the space $X_n$) in $X$ such that $A_n:=\left(d(x_1,x_k),d(x_1,x_2),\dots,d(x_{k-1},x_k) \right)$ is a $K$-relaxed polygon. Therefore
\[
B:=\left(f\circ d(x_1, x_k), \dots, f\circ d(x_{k-1},x_k) \right)
\]
does not form a $n$-relaxed polygon and thus $(X,f\circ d)$ fails to satisfy $n$-rpi. 
\end{proof}

It is also possible to investigate some results concerning families $P_{P,S}, P_{P,B}, P_{P,M}$ as well as $P_{S,P}, P_{B,P}, P_{M,P}$. The problem of obtaining analogous results for such classes of spaces is left as another open problem for the Readers.

This part of the paper would be somewhat lacking without any examples of functions in $P_{P}$. But before we provide any such characterization, we suggest some easy to check sufficient conditions for a function to be (P)-preserving. We will need the characterization of spaces satisfying $K$-rpi, which is due to Fagin et al. \cite{Fagin}.
\begin{theorem}[\textbf{Characterization of $K$-rpi spaces via metric bounds}]\label{faagiin}

Let $(X,d)$ be a semimetric space. The following are equivalent:
\begin{itemize}
\item[(i)] $(X,d)$ satisfies (P) with relaxation constant $K\geqslant 1$;
\item[(ii)] there exists a metric $d'$ on $X$ such that for all $x,y\in X$ 
\begin{equation}\label{equationes}
d'(x,y)\leqslant d(x,y) \leqslant K\cdot d'(x,y)
\end{equation}
\end{itemize}
\end{theorem}

Having this characterization at our disposal we can proceed with the following result.

\begin{lemma}[\textbf{Sufficient conditions for (P)-preservation}]

Let $f:[0,+\infty)\to[0,+\infty)$. If there exist $a,b\in \R$ such that $ax\leqslant f(x) \leqslant bx$ for all $x\geqslant 0$, then 
$f\in P_P$.
\end{lemma}
\begin{proof}
Assume that $f$ is bounded from both sides by linear functions $x\mapsto ax$ and $x\mapsto bx$. Let $(X,d)$ be any semimetric space satisfying condition (P). From Theorem \ref{faagiin} we obtain the existence of a metric $d'$ on $X$ such that the inequalities \eqref{equationes} hold. Since linear functions belong to $P_M$ due to being subadditive, increasing and amenable (see Lemma \ref{huek}), their compositions with $d'$ are also metrics. Thus we have
\[
ad'(x,y)\leqslant ad(x,y)\leqslant f(d(x,y)) \leqslant bd(x,y) = \frac{b}{a}\cdot ad'(x,y)
\]
Therefore $f\circ d$ is bounded from below by the metric $ad'(x,y)$ and from above by $\frac{b}{a} \cdot ad'(x,y)$. Due to Theorem \ref{faagiin} we have that $f(d(x,y))$ satisfies condition (P). Due to $(X,d)$ being arbitrary, $f\in P_P$.
\end{proof}

This allows us to construct this somewhat nontrivial example of a function from $P_P$.

\begin{example}
Let $f(x):=5x-4\lfloor x \rfloor$, where $\lfloor x \rfloor$ denotes the largest integer not exceeding $x$. 
\begin{figure}[H]
\includegraphics[width=0.7\textwidth]{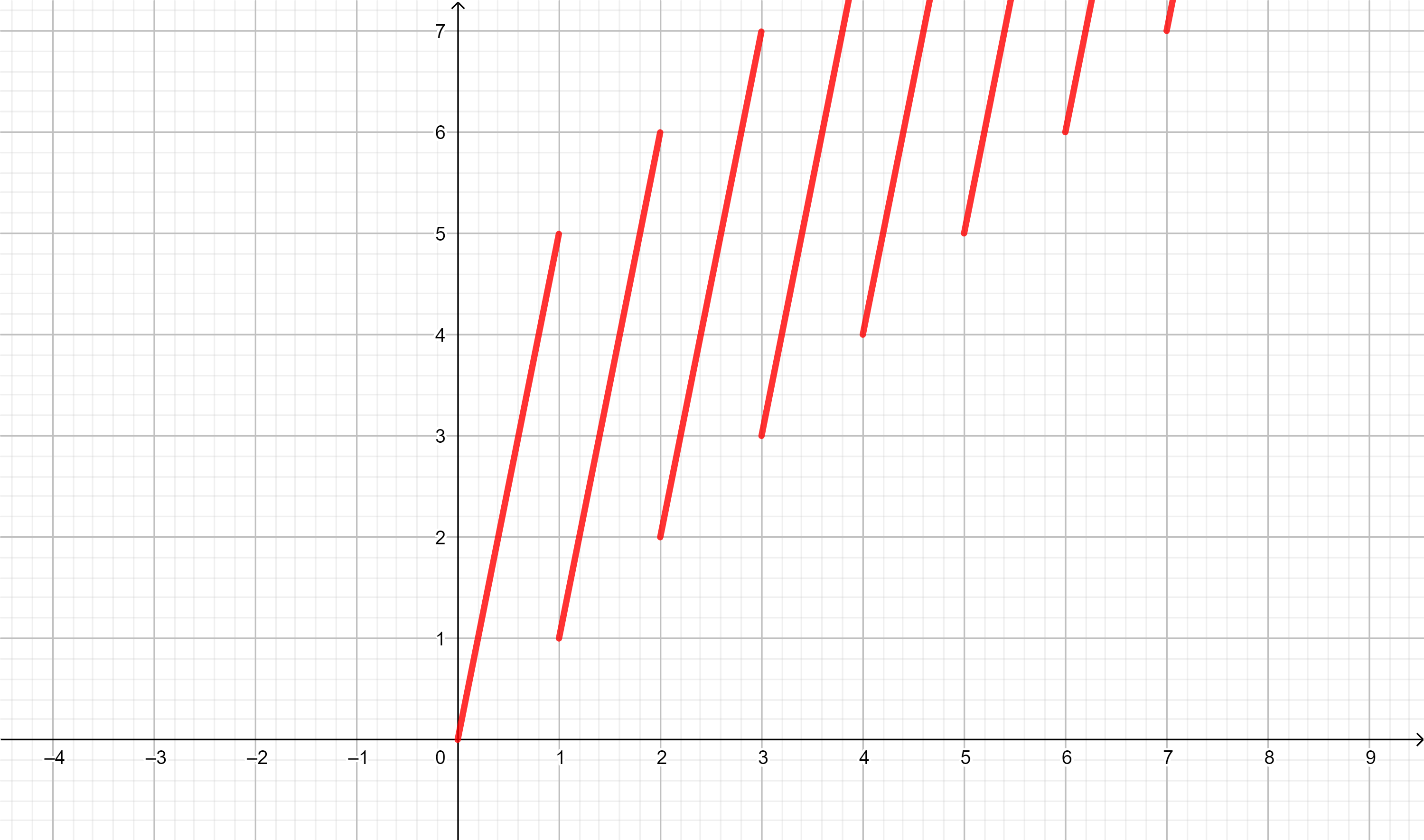}
\caption{A graph depicting \textit{sawtooth} function $f$.}
\end{figure}
Obviously $x\leqslant f(x) \leqslant 5x$, thus $f\in P_P$. One can also notice that $f$ does not belong to $P_M$, as it fails to be subadditive.  
\[
f(\frac{3}{2}) = \frac{15}{2} >f(1)+f\left(\frac{1}{2}\right) = 1+\frac{5}{2} = \frac{7}{2}.
\]
From Lemma \ref{imahusk} we obtain that $f\notin P_{M}$.
\end{example}

\begin{remark}
As a conclusion we obtain that although $P_M$ is not disjoint with $P_P$ (as identity belongs to both of those classes) we have $P_M\neq P_P$.
\end{remark}

\section{Acknowledgments}

	As usual,  I would like to express my utmost gratitude to both of my supervisors, that is Jacek Jachymski and Mateusz Krukowski for multitude of fruitful discussions which included priceless mathematical and writing tips. Lastly, I would like to acknowledge Piotr Nowakowski and Mateusz Lichman for adding some interesting comments on the topic.

\end{document}